\documentclass{article}

\usepackage[T1]{fontenc}
\usepackage[utf8]{inputenc}
\usepackage{lmodern}
\usepackage[english]{babel}
\usepackage[all,cmtip]{xy}
\usepackage{tikz-cd}
\usepackage{bm}
\usepackage{amsmath, amsthm, comment,color,graphicx}
\usepackage{amsfonts}
\usepackage{amssymb, mathtools}
\theoremstyle{definition} 

 \newtheorem{definition}{Definition}[section]

\theoremstyle{plain}      

 \newtheorem{proposition}[definition]{Proposition}
 \newtheorem{theorem}[definition]{Theorem}
 \newtheorem{corollary}[definition]{Corollary}
 \newtheorem{lemma}[definition]{Lemma}

\newtheorem*{theorem*}{Theorem}

\newcommand{\Z}{\mathbb Z}

\newcommand{\R}{\mathbb R}
\newcommand{\C}{\mathbb C}
\newcommand{\CP}[1]{\mathbb {CP}^#1}

\newcommand{\la}{\langle}
\newcommand{\ra}{\rangle}
\renewcommand{\P}{\mathbb {P}}

\newcommand{\bomega}{\bm \omega}

\renewcommand{\phi}{\varphi}
\newcommand{\bphi}{\bm \phi}

\newcommand{\SL}{\mathrm{SL}}
\newcommand{\GL}{\mathrm{GL}}
\newcommand{\PGL}{\mathrm{PGL}}

\newcommand{\PU}{\mathrm{PU}}
\newcommand{\SU}{\mathrm{SU}}
\newcommand{\SO}{\mathrm{SO}}
\newcommand{\Sp}{\mathrm{Sp}}
\newcommand{\I}{\mathrm{I}}
\newcommand{\II}{\mathrm{II}}
\newcommand{\III}{\mathrm{III}}
\newcommand{\IV}{\mathrm{IV}}

\newcommand{\Lg}{\mathfrak{g}}
\newcommand{\Lk}{\mathfrak{k}}
\newcommand{\Lm}{\mathfrak{m}}

\author{E. Falbel, A. Guilloux, P. Will}

\title{A Hilbert metric for bounded symmetric domains}

\begin{document}

\maketitle

\begin{abstract}
Bounded symmetric domains carry several natural invariant metrics, for 
example the Carathéodory, Kobayashi or the Bergman metric. We define 
another natural metric, from generalized Hilbert metric defined in 
\cite{FalbelGuillouxWill}, by considering the Borel embedding of the 
domain as an open subset of its dual compact Hermitian symmetric space 
and then its Harish-Chandra realization in projective spaces.

We describe this construction on the four classical families of bounded 
symmetric domains and compute both this metric and its associated 
Finsler metric. We  compare it to Carathéodory and Bergman 
metrics and show that, except for the complex hyperbolic space, those 
metrics differ.
\end{abstract}

\section{Introduction}

\subsection{Bounded symmetric domains and the generalized Hilbert metric}

Bounded domains in $\C^n$ are naturally equipped with invariant metrics 
under their biholomorphism group \cite{JP13}. Important invariant 
metrics include the Bergman, Kobayashi and Carathéodory metrics. 
Among bounded domains, the \emph{symmetric} ones are homogeneous and, 
for each point in the domain, admit a holomorphic involution whose 
differential is the negative identity map. Those domains are ubiquitous. 
Four families of bounded symmetric domains are named \emph{classical bounded 
symmetric domains} as their biholomorphism groups are contained in the 
family of classic semi-simple Lie groups. Explicit computations of 
invariant metrics for general domains are not easy but for symmetric 
domains one has closed formulas (Kobayashi and Carathéodory metrics coincide), 
see \cite{JP13} and Section \ref{coro:infinitesimal-metric-I}.

Those bounded symmetric domains correspond to Hermitian symmetric 
spaces, see Section \ref{ssec:BSD-HSS}. In the real case, symmetric spaces
can be embedded as open convex subsets of a real projective space, and 
as such inherit a Hilbert metric. This metric usually differs from the
natural Riemannian metric, except for real hyperbolic spaces \cite{BridsonHaefliger}.
The classical symmetric domains, as described in Section \ref{sec:BSD}, are
convex subsets of some complex affine space. As such, they could inherit
a Hilbert metric. However, the biholomorphism group does not act projectively;
this metric would not be invariant. 

The goal of this paper is to define an invariant Hilbert metric for classical 
bounded symmetric domains using \cite{FalbelGuillouxWill}, to compute it and 
to compare it to the classical metrics.
The definition of the generalized Hilbert metric in \cite{FalbelGuillouxWill}
involves a subset $\Omega$ of projective space and a subset $\Lambda$ of 
its dual projective space that verify a non-vanishing condition: 
for any $[\phi]\in \Lambda$ and $[v]\in \Omega$,
we have $\phi(v)\neq 0$, see Section \ref{section:Hilbert}. 
Then, for $\omega,\omega'\in \Omega$, the following formula is well-defined:
\begin{equation}\label{eq:Hilbert}
d_\Lambda (\omega,\omega') = \ln\left(\max \left\{
            \left|\left[\phi,\phi',\omega,\omega'\right]\right| \textrm{ for }
            \phi,\phi'\textrm{ in } \Lambda\right\}\right).
\end{equation} 
This defines a semi-metric\footnote{It is reflexive, satisfies the triangular 
inequality but may fail to separate points.} \cite{FalbelGuillouxWill} that 
is naturally invariant under projective transformations preserving both 
subsets. 

In order to apply the definition to the case of bounded symmetric domains and 
compute the resulting distance, one has to:
\begin{enumerate}
    \item realize the bounded symmetric domains as subsets $\Omega$ of a projective 
    space;
    \item define a suitable subset $\Lambda$ of the dual projective space;
    \item compute the formula (\ref{eq:Hilbert}) using the transitivity
    of the transformation group of the domain.
\end{enumerate} 
The first two points will rely on the classical theory of Borel and Harish-Chandra
embeddings, see Section \ref{sec:BSD}. The last point will follow from a 
normalization of pairs of points given by the singular value decomposition.
At the end, for a classical bounded symmetric domain $D$, we will have defined a 
metric $d_D$, whose infinitesimal metric is a Finsler metric. We will explicitly
describe those metrics in Sections \ref{sec:formula-type-I}, \ref{sec:formula-type-II-III}
and \ref{sec:formula-type-IV} through a case-by-case analysis of the four families.
This will give in particular the following Theorem \ref{thm:main}:
\begin{theorem*}
    For any classical bounded symmetric domain $D$, the semi-metric $d_D$ 
    is an actual metric, invariant by the automorphism group $\mathrm{Aut}(D)$ 
    and comes from a Finsler infinitesimal metric.
\end{theorem*}
	Except if $D$ is a complex hyperbolic space  this metric is neither the 
    Carathéodory nor the Bergman infinitesimal metric. The case of the complex hyperbolic spaces is particular. In this case, due to 
the rank one situation,
all metrics coincide up to a multiplicative factor. This was already observed in 
\cite[Lemma 4.3]{Zimmer-unitball} and \cite[Prop 3.4]{FalbelGuillouxWill}

 In order to better present the 
strategy before going into technicalities, we describe the organization following 
a simple example: the bidisc.

\subsection{A special case: the bidisc}\label{ssec:bidisc}

We explain the steps above on the example of the bidisc as a guideline, referring along the way to the general statements in the paper. Let
$D = U^2$ be the bidisc, where $U =\{z\in \C, \, |z|<1\}$. Then,
 $U$ has a unique 
$\PU(1,1)$-invariant metric which is the hyperbolic metric 
and coincides with the Carathéodory, the Bergman metric and 
the generalized Hilbert metric (see \cite{FalbelGuillouxWill}).
The automorphism group is generated by the exchange of the factors
and the action of $\PU(1,1)$ on either factor.

The Carathéodory and Bergman metric correspond respectively to the 
$L^\infty$ and $L^2$ metrics on $D$ (see for instance sections 56 
and 57 in Chapter 5 of \cite{Shabat}, or \cite{JP13}). The first step
of our strategy is to understand a projective embedding of $D$.

\subsubsection{Borel and Harish-Chandra embeddings}\label{ssec:bidisc1}

We give a down-to-earth presentation of the projective embedding of $D$. 
We refer the reader to \cite{wolf} and Section \ref{sec:Borel} below for more general information. The Plücker embedding of 
$\CP{1}\times \CP{1}$ in $\CP{3}$ is described in homogeneous coordinates by 
\begin{equation*}
 ([x:y],[u:v]) \longmapsto [xu:yu:xv:yv].
\end{equation*}
Its image is the quadric
$$Q = \{[x_0,x_1,x_2,x_3], x_0 x_3 - x_1x_2 = 0\}.$$
Viewing the disk as the subset $U\simeq \{[1:z],|z|<1\}$ of $\CP{1}$, we identify 
any $z\in U$ to $[1,z] \in \CP{1}$. We can restrict the Plücker embedding to 
$D=U^2$. This gives the map
    $$E: \left\{ \begin{matrix} D & \to & \CP{3}\\
        (z,w) & \longmapsto &[1:z:w:zw].\end{matrix}\right.
    $$
The biholomorphisms of $D$ extend to projective transformations of $\CP{3}$: 
one may check that the image of the biholomorphism group is the connected 
component in the orthogonal group of the quadratic form $x_0 x_3 - x_1x_2$. 
Note also that this quadratic form is non-degenerate, so identifies $\CP{3}$ 
to its dual projective space $(\CP{3})^\vee$. We call $F : D\to (\CP{3})^\vee$ the composition 
of the map $E$ and this duality. In coordinates, for any $(z,w)\in D$, 
the form $F(z,w)$, that we denote by $F_{(z,w)}$, is given up to a 
multiplicative constant by:
$$F_{(z,w)} : (x_0,x_1,x_2,x_3) \longmapsto x_0(zw) + x_3 - z x_2 - w x_1.$$

In general, one can realize bounded symmetric domains as subsets of a 
projective space. This comes from the fact that a bounded symmetric domain
is the image under the Borel embedding of a Hermitian symmetric space, 
and the Harish-Chandra embedding maps them into a projective space -- namely
a Grassmannian space.
Under this embedding, the automorphism group of $D$ is sent to a subgroup of
the group of projective transformations. Moreover, there are in fact a pair
of such embeddings, in two dual projective spaces. We review the necessary
background on classical Hermitian symmetric spaces and their 
links to bounded symmetric domains in Section \ref{ssec:BSD-HSS}; we then
describe the embeddings $E$ and $F$ in Section \ref{ssec:projective-embeddings}.
More details on the construction are then given in Section \ref{sec:Borel} for
the interested reader. Note that the maps $E$ and $F$ naturally extend to 
the boundary of $D$, and this holds in the general case, see Section 
\ref{ssec:projective-embeddings}.

\subsubsection{The Shilov boundary as the set $\Lambda$}\label{ssec:bidisc2}

The second step is to define a subset $\Lambda$ in the dual projective space.
We define $\Lambda$ as the image under the "dual" map $F$ of a classical and 
very important subset of the boundary of $D$: the \emph{Shilov boundary}, see 
Section \ref{sec:BSD}.
In the case of the bidisc, the Shilov boundary is the torus 
$\partial_S D:= \{(z,w),\, |z|=|w|=1\}$, see \cite{Clerc}.

We can check that, for all $p=(z,w)\in D$ and $\xi = (z_0,w_0)\in \partial_D S$, 
we have 
$$F_\xi(E(z,w)) = (z-z_0)(w-w_0) \neq 0.$$
So, linear forms in our $\Lambda$ do not vanish on the image $E(D)$. Hence, 
we can define the Hilbert metric $d_D$ by, for $p$, $p'$ in $D$:
$$d_D(p,p') = \ln \left( \max_{\xi,\xi' \in \partial_S D}\left|\frac{F_\xi(E(p))F_\xi'(E(p'))}{F_\xi(E(p'))F_\xi'(E(p))}\right|\right).$$

In general, the Shilov boundary of the classical domains are well-known, see Section 
\ref{ssec:BSD-HSS}. Using the maps $E$ and $F$ defined as above, we will always define the set $\Lambda$ as the image under $F$ of the Shilov boundary. We check in Section 
\ref{ssec:projective-embeddings} that the non-vanishing condition is always fulfilled.
We could have taken for $\Lambda$ the image by $F$ of the whole boundary of $D$. However,
due to the maximality properties of the Shilov boundary, see Definition \ref{def:Shilov},
it would not have changed the metric $d_D$. So we prefer to restrict to the smallest
$\Lambda$.

\subsubsection{Transitivity and computation of the Hilbert metric}\label{ssec:bidisc3}

In order to give an actual formula for $d_D$, we have to determine for which 
choice of $(\xi,\xi')\in \partial_S D$ the maximum is obtained. For the bidisc, this 
can be done directly, but we prefer to use the action of the automorphism group
to draw further the analogy with the general case. 

Indeed, given any two points $p$, $p'$ in the bidisc $D$, there is an automorphism 
that sends $p$ to $o:=(0,0)$ and $p'$ to $o'=(x,y)$ where $1>x\geq y\geq 0$. So we just need
to compute $d_D(o,o')$.  In the following computation, we denote by $\xi=(a,b)$ and $\xi'=(a',b')$ two points in the Shilov boundary $\partial_S D$.
\begin{eqnarray*}
    d_D (o,o') & = &\max_{\xi,\xi'\in \partial_S D} \ln \left|\frac{F_{\xi}(o)F_{\xi'}(o')}{F_{\xi}(o')F_{\xi'}(o)}\right|\\
    & = &\max_{\xi,\xi'} \ln \left|\frac{(-a)(-b)(x-a')(y-b')}{(x-a)(y-b)(-a')(b')}\right|.
\end{eqnarray*}
Recall that all the moduli of $a,b,a'$ and $b'$ are equal to $1$. One may separate the ones with $x$ and the ones with $y$. One can then determine that the maximum is attained for $(a,b)=(1,1)$ and $(a',b')=(-1,-1)$.
We recover this way the sum of the hyperbolic metrics on both factors:
\begin{align*}
    d_D (o,o') & = \ln \left|\frac{x+1}{x-1}\right| + \ln \left|\frac{y+1}{y-1}\right|\\
    &= d_U(z,z')+d_U(w,w').
\end{align*}
This discussion proves, for the bidisc, the formula for the Hilbert metric:
\begin{proposition}\label{prop:Hilbert-D2}
    The generalized Hilbert metric on $D^2$ is the the sum of the hyperbolic metrics on each factor.  
\end{proposition}

In the general case, one uses classically the singular value decomposition to describe pair of points in the classical bounded symmetric domains, see Section \ref{sec:SVD}. The determination of the maximum and its consequences are then done in the last Section \ref{section:Hilbert}. 

\subsection*{Acknowledgements}
We thank Martin Deraux for discussions.  The three authors are partially funded by  ANR-23-CE40-0012.

\section{Classical bounded symmetric domains \label{sec:BSD}}

Let $\Omega\subset \C^n$ be a bounded domain.
There are several situations where a natural metric can be associated to it.
A very general definition is that of the Bergman metric on any bounded domain.
Other definitions of invariant metrics include the Carathéodory and Kobayashi metrics.
The construction is such that the biholomorphism group is contained in the 
isometry group of those metrics. The particular case where $\Omega$ is a bounded 
homogeneous domain has been studied for a long time.  
It contains the important class of non-compact Hermitian symmetric spaces.  
These Riemannian spaces, classified by Cartan, can be embedded as bounded domains 
which contain the origin, are stable under 
the circle action  and which turn out to be convex (see \cite{wolf} for a 
thorough exposition).  In the case of bounded symmetric domains  the Carathéodory 
and Kobayashi metrics coincide.

The group of biholomorphisms of a bounded symmetric domain is transitive and 
can be extended to the boundary. But its action on the boundary is not transitive 
except in the case of the complex ball.  On the other hand, the isotropy group 
at the origin acts by linear maps of $\C^n$ and, moreover, it acts transitively 
on the Shilov boundary of $\Omega$:
\begin{definition}\label{def:Shilov}
Let $\Omega\subset \C^n$ be a bounded domain. The Shilov boundary of 
$\Omega$ is the smallest closed subset  $S\subset \partial \Omega$ such 
that, for every holomorphic function $f$ defined on $\Omega$ which extends 
continuously to the boundary,  we have
$$
\vert f(z)\vert \leq \max \{  \vert f(w)\vert\, , w\in S \}.
$$
\end{definition}
We will not use much properties of the Shilov boundary, except its actual description
for bounded symmetric domains, see Section \ref{ssec:BSD-HSS}. For a general
presentation, we refer to \cite{Clerc}.

We now present the classical cases, together with the associated Hermitian symmetric
spaces and the Borel and Harish-Chandra embeddings.

\subsection{Bounded symmetric domains and Hermitian symmetric spaces}\label{ssec:BSD-HSS}

We now review the needed materials to extend the two dual projective embeddings 
of the bidisc (see Section \ref{ssec:bidisc1}) to all four families of classical bounded symmetric domains.

\subsubsection{Hermitian symmetric spaces}

General references for this section are \cite{Helgason, wolf} and \cite{Clerc} 
for the Shilov boundary.
A Riemannian symmetric space $X$ is a Riemannian manifold such that, 
at each point $x\in X$, there exists an isometric involution 
$\sigma_x$ with $d{\sigma_x}=-Id$. In the case $X$ is a Hermitian 
manifold and $\sigma_x$ is a holomorphic map we say that
$X$ is a Hermitian symmetric space.  

Any simply connected Hermitian symmetric space is a product 
$X= X_e\times X_-\times X_+$ where $X_e$ is an Euclidean space, 
$X_-$ is a compact Hermitian symmetric space and $X_+$ is a non-compact 
Hermitian symmetric space (Proposition 4.4 in \cite[Chapter VIII]{Helgason}).
Moreover, $X_-$ and $X_+$ themselves are products of irreducible symmetric 
spaces of compact and non-compact type respectively, with each factor being of 
the form $G/K$ where $G$ is a connected simple Lie group with trivial center 
and $K$ a maximal compact subgroup of $G$. Be aware that one does not 
consider Euclidean spaces as non-compact type symmetric spaces.  
Hermitian symmetric spaces of compact type are projective algebraic manifolds.

There exists a duality between Hermitian symmetric spaces of compact and 
non-compact type. The construction goes as follow: start with $X_0=G_0/K$ 
an irreducible non-compact Hermitian space and write the Lie algebra of its
holomorphic automorphism group as $\Lg_0=\Lk\oplus \Lm$; here $\Lk$ is a 
maximal compact sub-algebra and the decomposition corresponds to the $1$ 
and $-1$ eigenspaces of a Cartan involution of $\Lg_0$. The {\it compact dual} 
of $\Lg_0$ is $\Lg_c=\Lk\oplus \sqrt{-1}\,\Lm$ (we write $G_c$ for the 
corresponding Lie group).  In fact, $\Lg_c$ is a maximal compact sub-algebra  
of the complexification $\Lg^\C$.  One obtains that $X=G_c/K$ is a compact 
Hermitian symmetric space which is called the dual of  $X_0=G_0/K$.
The Borel embedding theorem states that a non-compact Hermitian symmetric 
space has an embedding as an open subset of its compact dual. 
\begin{theorem*}[Borel embedding theorem] Let $G^\C$ be the connected 
    Lie group corresponding to $\Lg^\C$.  Then the compact Hermitian 
    space $X_c=G_c/K$ is biholomorphic to $G^\C/P$, where $P\subset G^\C$ 
    is a complex subgroup. The embedding $G_0\to G^\C$ induces an open 
    embedding $X_0\to X_c$ realizing $X_0$ as an open subset in its 
    compact dual $X_c$.
\end{theorem*}
An important fact is that the connected automorphism $G_0$ of the non-compact 
Hermitian symmetric space is realized as a subgroup of automorphisms of 
the projective manifold $X_c$. Therefore it is realized as a subgroup of projective transformations.

\subsubsection{Bounded symmetric domains}

The link with bounded symmetric domains is that all Hermitian symmetric 
spaces of non-compact type can be realized as bounded symmetric domains.
\begin{definition}
A bounded domain $D\subset \C^n$ is a symmetric domain if, at each point 
$z\in D$, there exists a holomorphic involution $\sigma_z$ such that 
$d\sigma_z=-Id$.
\end{definition}
Each bounded domain carries a Bergman metric which is invariant under 
the automorphism group of $D$ (see \cite[Chapter VIII]{Helgason} or 
\cite[Chapter V]{Shabat}).  The Harish-Chandra theorem states that 
any Hermitian symmetric space is identified to a bounded symmetric 
domain $D$ equipped with its Bergman metric, and conversely:
\begin{theorem}[Harish-Chandra] A Hermitian symmetric space of 
    non-compact type is biholomorphic to a bounded symmetric domain 
    in an euclidean space which is embedded as a dense domain in the 
    compact dual Hermitian symmetric space.
\end{theorem}
Among the Hermitian symmetric spaces, four families are special: 
the {\it classical Hermitian symmetric spaces}.  We will now describe 
the bounded symmetric domains corresponding to these Hermitian symmetric 
spaces, as well as their Shilov boundaries. We will provide in the next 
paragraph a short description of the way this realization is done.
To define these domains, we consider the groups:
\begin{eqnarray*}
\Sp(n,\R) = & \{M\in \GL(2n,\R)\ \vert\ {{M}}^TJ_nM=J_n\}\\
\SO^*(2n) = & \{M\in \SL(2n,\C)\ \vert\ M^TM=I_{2n}, \ {\overline{M}}^TJ_nM=J_n\ \},
\end{eqnarray*}
where $J_n= \begin{pmatrix} 0 & \I_n \\
						   -\I_n &0 
		   \end{pmatrix}$. 

For a Hermitian matrix $A$, the notation $A>0$ means that it is positive definite.

\begin{enumerate}
\item[{\bf Type}]{\bf $\I$:} For $1\leq p\leq q$, the symmetric space 
is $X^\I_{p,q} = \SU(p,q)/\mathrm{S}(\mathrm{U}(p)\times \mathrm{U}(q))$. 
The associated domain and its Shilov boundary are respectively given by: 
\begin{align}\label{eq:type-I}
D^\I_{p,q} =&\left\{\  Z\in M_{p,q}(\C)\ \vert \ I_{q}-Z ^*Z> 0\ \right\}\nonumber\\
\partial_S D^\I_{p,q}=&\left\{\  Z\in M_{p,q}(\C)\ \vert \ ZZ^*-I_p= 0\ \right\}.
\end{align}

\item[{\bf Type}]{\bf $\II$:} For $n\geq 5$, the symmetric space is 
$X^{\II}_{n} =\SO^*(2n)/\mathrm{U}(n)$; the associated domain and its 
Shilov boundary are given by the following construction, depending on 
the parity of $n$. Let $Z_0$ be the block-diagonal matrix with blocks 
$\left(\begin{smallmatrix} 0&1\\-1&0\end{smallmatrix}\right)$ and of 
rank $n$ if $n$ is even, $n-1$ if $n$ is odd. Then $Z_0Z_0^*-I_n$ is 
either $0$, if $n$ is even, or the diagonal matrix with diagonal entries 
$(0,\ldots, 0, -1)$ if $n$ is odd. The Shilov boundary is the orbit of 
$Z_0$ under the transformation group.
\begin{align}\label{eq:type-II}
	D^{\II}_{n}=&\left\{\  Z\in M_{n,n}(\C)\ \vert \ I_{n}-Z^* Z> 0, 
        \ Z^T=-Z\ \right\} \subset D^\I_{n,n}\\
	\partial_S D^{\II}_{n}=&\left\{\  Z\in M_{n,n}(\C)\ 
    \vert \ Z Z^*-I_n\ = Z_0Z_0^*-I_n, \ Z^T=-Z\ \right\}.\nonumber
\end{align}
Note that $\partial_S D^\II_n \subset \partial_S D^\I_{n,n}$ iff $n$ is even.

\item[{\bf Type}]{\bf $\III$:} For $n\geq 2$, the space is 
$X^\III_n = \Sp(n,\R)/\mathrm{U}(n)$; the associated domain and 
its Shilov boundary are given by: 
\begin{align}\label{eq:type-III}
	D^{\III}_{n}=&\left\{\  Z\in M_{n,n}\ \vert \ I_{n}- Z^* Z> 0,
     \ Z^T=Z\ \right\}\subset D^\I_{n,n}\\
	\partial_S D^{\III}_{n}=&\left\{\  Z\in M_{n,n}(\C)\ \vert \ Z Z^*- I_n=0,
     \ Z^T=Z\ \right\}\subset \partial_S D^\I_{n,n}.\nonumber
\end{align}

\item[{\bf Type}]{\bf $\IV$:} For $n\geq 2$, the space is 
$X^\IV_n = \SO^0(n,2)/\SO(n)\times \SO(2)$; the associated domain 
and its Shilov boundary are given by: 
\begin{align}\label{eq:type-IV}
D^\IV_n =& \left\{z\in M_{1,n}(\C)\ \vert \ z^*z<2\textrm{ and } 
z^*z<1+\left|\frac 12 z^Tz\right|^2\right\}\\
\partial_S D^{{\IV}}_{n} = &\left\{ e^{i\theta}x\textrm{ for }
x\in M_{1,n}(\R),\ x^Tx=2\textrm{ and }\theta \in \R\right\}.\nonumber
\end{align}
\end{enumerate}
We now describe explicitly the projective embeddings, as in the case of
the bidisc, see Section \ref{ssec:bidisc1}.

\subsubsection{Projective embeddings}\label{ssec:projective-embeddings}

We recall here classical material 
(see \cite{mok-MetricRigidity},\cite{these-Wienhard}) about Borel and 
Harish-Chandra embeddings, and their descriptions using Grassmannians 
and the Plücker map. For the interested reader we will provide a 
review of these facts in the next section.

For integers such that $p+q=n$, the two projective spaces 
$\P(\Lambda^p\C^n)$ and $\P(\Lambda^q\C^n)$ are dual to one another. 
Indeed, the choice of the determinant function defines a pairing 
$\la \cdot, \cdot\ra$ which is defined on pure tensors by
\begin{eqnarray}
 \Lambda^p\C^n\times \Lambda^q\C^n & \longrightarrow & \Lambda^n\C^n \overset{\det}{\sim} \C\nonumber\\
 (A,B) & \longmapsto & A\wedge B = \det(A,B). \label{eq:identification-dual}
\end{eqnarray}

This allows us to identify $\Lambda^q\C^n$ to $(\Lambda^p\C^n)^*$.  
In turn, we obtain a natural identification\footnote{After projectivization, 
the identification does not depend on the choice of the determinant.}
\begin{equation}\label{eq:duality}
    \P(\Lambda^q\C^n)\sim \P(\Lambda^p\C^n)^*.
\end{equation}
The following proposition sums up what we need from the Borel and Harish-Chandra
embeddings:
\begin{proposition}\label{prop:proj-embed-all}
    For each classical bounded symmetric domain $D$ in the above list, 
    there exists $(n,p,q)$ with $p+q=n$, and two embeddings $E$ and $F$
    \begin{equation*}
        E : D \longrightarrow \P(\Lambda^q\C^n),\, F :  
        D \longrightarrow \P(\Lambda^p\C^n),
    \end{equation*}
    such that:
    \begin{enumerate}
        \item both maps extend continuously to the boundary $\partial D$,
        \item the action of $\mathrm{Aut}(D)$ translates into the action of a 
        subgroup of $\PGL(n,\C)$ preserving the image of $D$,
        \item if $\xi \in \partial D$ then $\la F(\xi), E(\xi)\ra=0$,
        \item if $z \in D$, and $\xi\in\partial D$ then $\la F(\xi), E(z)\ra\neq 0$,
    \end{enumerate}
\end{proposition}
We still denote by $E$ and $F$ the extensions to the boundary.

The two maps $E$ and $F$ are easily described for domains of types $\I,\II$ and 
$\III$, as we do now. Given a matrix  $M\in{\rm Mat}(n,k,\C)$ we still denote by 
$M\in\Lambda^k\C^n$ the exterior product of its columns. We denote by $[M]$ 
its class in the projective space $\P(\Lambda^k\C^n)$ (whenever ${\rm rk } M =k$). 
Using this notation, write $n=p+q$ where $p$ and $q$ are as in \eqref{eq:type-I}.    
Define the following two maps. 
\begin{equation}
    \begin{matrix}
        E_q :& D^\I_{p,q} & \longrightarrow & \P(\Lambda^q\C^n)\\
        &   Z          & \longmapsto     & \begin{bmatrix} Z \\ I_q\end{bmatrix}
    \end{matrix}
    \quad \mbox{ and }\quad
    \begin{matrix}
        F_p : & D^\I_{p,q} & \longrightarrow & \P(\Lambda^p\C^n)\\
        &   Z          & \longmapsto     & \begin{bmatrix} I_p \\ Z^*\end{bmatrix}
    \end{matrix}
\end{equation}
Then the maps $E$ and $F$ of Proposition \ref{prop:proj-embed-all} are given in 
Type $\I$, $\II$, $\III$ by these maps, for suitable values of $p$ and $q$:
\begin{proposition}\label{prop:proj-embedding}
\begin{enumerate}
 \item The two maps $E_q$ and $F_p$ are embeddings, and they extend to the boundary of $D^\I_{p,q}$. 
 \item The action of ${\rm Aut}(D^\I_{p,q})$ translates into the natural action of ${\PU}(p,q)$ on  $\P(\Lambda^q\C^n)$ and  $\P(\Lambda^p\C^n)$.
 \item For domains of types $\II$ and $\III$, $1.$ and $2.$ hold when restricting $E_n$ and $F_n$ to the subdomains $D^{\II}_{n}\subset D^\I_{n,n}$ and $D^{\III}_{n}\subset D^\I_{n,n}$
\end{enumerate}
\end{proposition}
The following easy proposition checks the non-vanishing property (see Section 
\ref{ssec:bidisc2}), that will allow us to define the Hilbert metric later on.
\begin{proposition}\label{prop:duality-explicit}
For any $Z$ and $Z'$  in the topological closure $\overline{D_{p,q}^\I}$, we have
$$F_p(Z)\wedge E_q(Z') = \det(I_p-Z'Z^*).$$
\end{proposition}
\begin{proof}
 Using our notation,
 $$F_p(Z)\wedge E_q(Z') = \det\begin{pmatrix}I_p & Z'\\Z^* & I_q\end{pmatrix}.$$
 The result follows directly from the fact that for any invertible matrix $D$
  \begin{align}\label{eq:determinant}
\det  \begin{pmatrix} A & B \\
							C & D
\end{pmatrix} = \det (D)\det (A-BD^{-1} C).
\end{align}
\end{proof}
The following subsection reviews in depth Borel and Harish-Chandra embeddings, to prove
the two previous propositions and describe de Type $\IV$. Readers convinced by the previous
statements and less interested in Type $\IV$ may prefer to skip it and directly jump
to Section \ref{sec:SVD}.
 
\subsection{Borel and Harish-Chandra embeddings}\label{sec:Borel}
The link between the classical Hermitian symmetric spaces and symmetric 
bounded domains may be accomplished considering an action of the group on a 
Grassmannian, and then using an affine chart on the Grassmannian.
These two steps are respectively the Borel embedding and the 
Harish-Chandra embedding. We present explicitly these embeddings for 
the classical case since we use them afterwards. We treat types $\I$, 
$\II$ and $\III$ together and then present Type $\IV$.

\subsubsection{Types $\I$, $\II$, $\III$}
Type $\I$ Hermitian symmetric spaces are the quotients 
$X^\I_{p,q} = SU(p,q)/S(U(p)\times U(q))$. Since we consider $p$ 
and $q$ fixed in this section, we will drop the dependence in $p$ 
and $q$ and simply denote the space by $X^\I$.
Let us denote by $G(k,n)$ the Grassmannian of $k$-spaces in $\C^n$, 
where $n=p+q$.
The Borel embeddings of $X^\I$ realize it as subspaces of the 
Grassmannians $G(p,n)$ or $G(q,n)$. 
To describe these embeddings, let us endow $\C^n$ with a Hermitian 
form  $h$ of signature $(p,q)$ given in the canonical basis 
$(e_i)_{1\leqslant i\leqslant n}$ of $\C^{p+q}$ by 
\[
    h(z,z)=\sum_{i=1}^p |z_i|^2 -\sum_{j=1}^q|z_{p+i}|^2.
\] 
We denote by $I_{p,q}$ the matrix of $h$ in the canonical basis of 
$\C^n$, and we decompose $\C^n$ as the orthogonal direct sum 
$\C^n = W^+\oplus W^-$ (where $W^+=\la e_1,\cdots,e_p\ra$ and 
$W^- = \la e_{p+1},\cdots,e_{p+q}\ra$). The Borel embeddings of 
$X^\I$ are the following two (well-defined) maps
\begin{eqnarray}\label{eq:embedp}
X^\I & \longrightarrow & G(p,n) \nonumber\\
         x = \lbrack g \rbrack     & \longmapsto & g\cdot W^+ = W^+_x
\end{eqnarray}
and
\begin{eqnarray}\label{eq:embedq}
 X^\I & \longrightarrow & G(q,n) \nonumber\\
       x =  \lbrack g \rbrack    &\longmapsto & g\cdot W^- = W^-_x,
\end{eqnarray}
where $[g]$ denotes the class of $g$ in the quotient 
$\SU(p,q)/{\rm S}({\rm U}(p)\times {\rm U}(q))$. Note that there is a 
natural duality between $G(p,n)$ and $G(q,n)$, given by orthogonality for $h$. 
In particular $W^+_x$ and $W^-_x$ are orthogonal. The actions of the group 
$SU(p,q)$ on the Grassmannians decomposes  $G(p,n)$ and $G(q,n)$ into orbits. 
The images of $X^\I$ by the maps \eqref{eq:embedp} and \eqref{eq:embedq} 
are respectively $G^+(p,n)$ and $G^-(q,n)$, the sets of positive $p$-subspaces 
(resp. negative $q$-subspaces), namely:
\[
    G^+(p,n) = \lbrace W\in G(p,n), h_{|W}>0 \rbrace \subset G(p,n)
\]
and
\[ 
    G^-(q,n) = \lbrace W\in G(q,n), h_{|W}<0  \rbrace \subset G(q,n).
\]

The Harish-Chandra realizations of $X^\I$ as bounded domains follow 
from choosing affine charts on the two Grassmannians, which contain  
$G^+(p,n)$ and $G ^-(q,n)$ respectively.  Let us explain how this work 
for $G^-(q,n)$ (the situation is symmetric for $G^+(p,n)$). Observe first 
that any subspace $W_x\in G^-(q,n)$ is transverse to $W^+$ (meaning 
$ W_x\cap W^+ = \{0\}$), and therefore is the graph of a unique linear 
map $L_x : W^-\longrightarrow W^+$. Using the bases of $W^-$ and $W^+$ 
coming from the canonical basis, we identify $L_x$ to a matrix 
$Z_x\in{\rm Mat}(p,q,\C)$.  For any such matrix $Z\in{\rm Mat}(p,q,\C)$, 
the open condition expressing that $Z$ corresponds to a negative type 
$q$-subspace is given by imposing that the graph of the linear map 
given by $Z$ has negative type, that is
$$ \forall w \in W^-,\, h(w+Zw,w+Zw)<0.$$ 
This is equivalent to 
\[
\begin{pmatrix} Z^*, I_q\end{pmatrix} I_{p,q}  \begin{pmatrix}Z\\I_q\end{pmatrix}  = Z^*Z-I_q <0.
\]
The map
\[
 x \longmapsto Z_x
\]
gives the first Harish-Chandra realization of $X^\I$  as the domain 
$D^\I_{p,q}$ described in \eqref{eq:type-I}.  Using the orthogonality 
relation between $G^-(q,n)$ and $G^+(p,n)$, we obtain that the other 
Harish-Chandra realization of $X^\I$ is given by
\begin{equation}\label{eq:HC2}
 x \longmapsto Z_x^*
\end{equation}
The image of $X^\I$ by the map \eqref{eq:HC2} is the domain 
$\{ S\in {\rm Mat}(q,p,\C), S S^*-I_p <0\}$, which is 
biholomorphic to $D^\I_{p,q}$. 

To sum up, as displayed in Figure \ref{fig:diagram}, the first Borel and Harish-Chandra embeddings associate 
to a point $x\in X^\I_{p,q}$ a $q$-subspace $W_x^-$ and a matrix $Z_x$. 
These two are related by
$$W_x^-={\rm Span}\begin{pmatrix}Z_x\\I_q\end{pmatrix},$$
where we mean by ${\rm Span}(M)$ the subspace generated by the columns 
of the matrix $M$.
The other embedding associate to $x$ the matrix $Z_x^*$ and the $p$-subspace 
$$W_x^+=(W_x^-)^\perp ={\rm Span}\begin{pmatrix}I_p\\Z_x^*\end{pmatrix},$$ 
\begin{figure}[ht]\label{fig:diagram}
    \begin{center}
        \begin{tikzcd}[]
            & & x\in X^\I_{p,q} 
                \arrow[dddl, "B" description, leftrightarrow] 
                \arrow[ddr, "B" description, leftrightarrow] & &\\
            &&&&\\
            & & & W_x^+ \arrow[dll, "\perp" description, leftrightarrow, blue] 
                        \arrow[ddl, "HC" description, leftrightarrow] 
                        \arrow[ddr, "P" description] &\\
            & W_x^-  &&&\\
            &&Z_x^* \arrow[dll, "*" description, leftrightarrow, blue] & & E_p(x) \arrow[dll, "\perp" description, leftrightarrow, blue]\\
            Z_x \arrow[from=uur,"HC" description, leftrightarrow]
                \arrow[rr, "E" description, red]
                \arrow[urrrr, "F" description, red]
                    & & E_q(x) \arrow[from=uul,"P" description,crossing over] & &\\
        \end{tikzcd}
        \caption{A synthetic view of the Borel and Harish-Chandra embeddings and the maps $E$ and $F$ (in red) for Type $\I$.\\
        Here, we have $W_x^+ \in G(p,n)$, $W_x^- \in G(q,n)$, $Z_x \in D^\I_{p,q}\subset M_{p,q}$ and $Z_x^* \in M_{q,p}$, $F(x) = E_q(x)\in \P(\Lambda^q\C^n)$ and $E(x) = E_p(x)\in \P(\Lambda^p\C^n)$.
        The black arrows are the different steps of Borel, Harish-Chandra and Plücker embeddings, denoted by respectively the letters $B$, $HC$ and $P$.\\
        The three blue arrows are given by orthogonality: taking the adjoint on the bottom left, orthogonality for the Hermitian structure in the middle, and orthogonality in duality in the bottom right.
        }
    \end{center}
\end{figure}
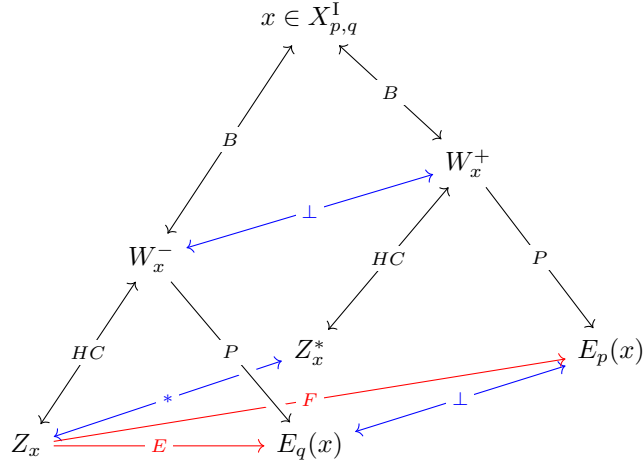

The embeddings for the domains of type $\II$ and $\III$ are obtained by 
restrictions of the embeddings for domains of type one to the appropriate sub-spaces.

\subsubsection{Type $\IV$\label{section:typeIV}}
Following the description we have sketched for types $\I$, $\II$ and $\III$, 
the type $\IV$ symmetric space 
$\SO_0(n,2)/(\SO(n)\times \SO(2))=\SO(n,2)/{\rm S}({\rm O}(n)\times {\rm O}(2))$ 
can be realized in two dual ways as open subsets of real Grassmannians on  
$\R^{n+2}$ equipped with a $(2,n)$ quadratic form: either as $G_\R^-(2,n+2)$ 
or as $G_\R^+(n,n+2)$. The latter are respectively the open subsets of the 
real Grassmannian of 2-planes (resp. $n$-subspaces) in $\R^{n+2}$ that are 
negative (resp. positive). However, to realize this family of symmetric spaces 
as complex bounded symmetric domains, it is customary to take a slightly 
different point of view which we describe here.

Let $q$ be the quadratic form of signature $(n,2)$ on $\R^{n+2}$ given in 
the canonical basis by
$$q(x)=x_1^2+\cdots x_n^2-x_{n+1}^2-x_{n+2}^2.$$
The form $q$ extends naturally to a complex bilinear form $q$ on the 
complexification $\C^{n+2}$, which we still denote by $q$. It also gives 
rise to a Hermitian form $h$ of signature $(n,2)$,  obtained by setting
$$h(z,w)=q(z,\overline{w}),\, \forall z,w\in\C^n.$$
The group $SO_0(n,2)$ preserves both $q$ and $h$. 

Let us denote by $P_0$ the 2-plane in $\R^{n+2}$ given by 
$P_0=\la e_{n+1}, e_{n+2}\ra$, by $V_0=\la e_1\ldots e_n\ra$ its orthogonal 
complement in $\R^{n+2}$ and chose an orientation of $P_0$.  The transitive 
action of the group $SO_0(n,2)$ on the set of negative planes of $\R^{n+2}$  
defines unambiguously an orientation on each negative plane of $\R^{n+2}$. 
 
The plane $P_0$ has a unique complex structure $J$ preserving $h$ and such 
that for all vectors $x\in P_0\setminus\{0\}$, the pair $(Jx,x)$ is 
positively oriented. Choosing the orientation so that $(e_{n+1},e_{n+2})$ 
is positive, this complex structure is given by
    $$J(e_{n+1})=-e_{n+2},\mbox{ and }J(e_{n+2})=e_{n+1}.$$
The action of $J$ extends linearly to the complexification 
$P_0^\C\subset \C^{n+2}$. Let $L_0$ be the $i$-eigenline of 
$J$ in $P_0^\C$, which is spanned by $e_{n+1}+ie_{n+2}$.

The line $L_0$ is negative for $h$ and isotropic for $q$. The previous 
construction extends to all negative  2-planes in $\R^{n+2}$ by the 
action of $SO_0(n,2)$. From this discussion, it follows that the set of 
negative planes in $\R^{n+2}$ is in bijection with the set of complex 
lines in $\C^{n+2}$ that are negative for $h$ and $q$-isotropic (this 
is Proposition 6.1 in Appendix 6 of \cite{Satake}). The converse bijection 
is obtained by noting that given a complex line $L$ in  $\C^{n+2}$, which 
is $q$-isotropic and $h$-negative, then the complex plane 
$W=L\oplus \overline{L}$ is the complexification of a negative plane in $\R^{n+2}$.

Now, to obtain an explicit parametrization, we first do a linear change of  
coordinates by transforming  $e_{n+1}$ and $e_{n+2}$ respectively into 
$\frac{1}{\sqrt{2}}(e_{n+1}-ie_{n+2})$ and $\frac{1}{\sqrt{2}}(e_{n+1}+ie_{n+2})$. 
In these coordinates, we observe  that a complex line $L=\C z$ is 
$q$-isotropic and $h$-negative  if and only if $z$ satisfies 
\begin{eqnarray*}
\sum_{k=1}^nz_k^2 & = & 2z_{n+1}z_{n+2}\\
\sum_{k=1}^n |z_k|^2 & < & |z_{n+1}|^2+|z_{n+2}|^2.
\end{eqnarray*}
These two conditions define an open subset of a quadric in $\C P^{n+1}$, 
whose intersection with the affine chart $\{z_{n+2}=1\}$ is parametrized by
    $$\mathcal{O} = 
    \left\{ \bigl[z_1:\ldots :z_n:\dfrac{1}{2}\sum_{k=1}^nz_k^2:1\bigr],\, 
    \sum_{k=1}^n |z_k|^2 < 1 +
    \Bigl\vert\dfrac{1}{2}\sum_{k=1}^nz_k^2 \Bigr\vert^2\right\}.$$
Note that all points in $\mathcal{O}$ satisfy $|\sum_{k=1}^nz_k^2|\neq 2$ 
for otherwise we would have simultaneously $\sum_{k=1}^n|z_k|^2<2$ and 
$\left|\sum_{k=1}^nz_k^2\right|=2$.   
The complex line $L_0$ we started from corresponds to the point with 
homogeneous coordinates $[0:\ldots : 0 : 0 :1]$, and by connectedness 
of $SO_0(n,1)$, its orbit is contained in the subset 
  \[
    \mathcal{O}_0 = \left\{ 
        \bigl[z_1:\ldots :z_n:\dfrac{1}{2}\sum_{k=1}^nz_k^2:1\bigr],\, 
        |z^Tz|<2,\,z^*z<1+\Bigl\lvert\frac12z^Tz\Bigr\rvert^2
        \right\}
    \]
The orbit of $L_0$ is in fact {\it equal to} $\mathcal{O}_0$:  we refer 
to pages 76 and 77 of \cite{mok-MetricRigidity} for details.

As in the case of domains of types $\I$, $\II$ and $\III$, one can 
describe the dual realization and see $SO_0(n,2)/SO(n)\times SO(2)$ as 
a subset of the Grassmannian of hyperplanes in $\C^{n+2}$. This is done 
via the duality associated to the Hermitian form $h$. As an example, 
the hyperplane associated to the complex line $L_0$ is $H_0=L_0^{\perp_h}$ 
which is nothing but the hyperplane spanned by $(e_1,\ldots,e_n, e_{n+2}+ie_{n+1})$, 
that is the hyperplane spanned by $V_0$ and the $-i$ eigenline of $J$ 
restricted to $P_0^\C$. Again, this extends to all $q$-isotropic 
$h$-negative complex lines by the action of $SO_0(n,2)$.

To sum-up the previous discussion, the two dual Borel realizations of 
the symmetric space $SO_0(n,2)/SO(n)\times SO(2)$ as subsets of complex 
Grassmannians are respectively:
    \[
        {\rm Gr}(1,n+2)_0^-=\{ L\in G(1,n+2),\,  L \mbox{ is $q$-isotropic and 
    $h$-negative}.\}
    \]
and 
    \[
    {\rm Gr}(n+1,n+2)_0^+=\{H\in G(n+1,n+2), H^{\perp_h} 
    \mbox{ is $q$-isotropic and 
    $h$-negative}\}. 
    \]
To prove Proposition \ref{prop:proj-embed-all} in this case, it remains 
to describe the maps $E$ and $F$ in terms of the previous discussion. The 
map $E$ is defined as follows. 
\begin{eqnarray*}
     E :  D_n^{\IV} &\longrightarrow &{\rm Gr}(1,n+2)_0^- \\
     z=(z_1,\ldots, z_n) &\longmapsto& \C\cdot w,   
\end{eqnarray*}
where $w$ is the vector defined by 
\begin{equation}\label{eq:def-E-IV}
    w=\begin{bmatrix}z_1 \\ \vdots \\ z_n \\ z_{n+1}\\1 \end{bmatrix}
    \mbox{ with } z_{n+1}=\dfrac{1}{2}\sum_{i=1}^nz_i^2.
\end{equation}
Using the same notation as above, the map $F$ is given by 
\begin{eqnarray}\label{eq:def-F-IV}
 F :  D_n^{\IV} &\longrightarrow &{\rm Gr}(n+1,n+2)_0^- \nonumber \\
     z=(z_1,\ldots, z_n) &\longmapsto& 
     {\rm Span}
     \begin{pmatrix} 
     &  I_{n}  & & 0\\
     \overline{z_1} & \ldots & \overline{z_n} & -\dfrac{1}{2}\sum_{k=1}^n \overline z_k^2\\
     0			  & \ldots  &	0		&  1
     \end{pmatrix},   
\end{eqnarray}
which we verify by checking that each column of the matrix is orthogonal 
to $w$ with respect to the Hermitian product.

The analog of proposition \ref{prop:duality-explicit} is the following statement.
For any $z=(z_1,\ldots,z_n)$ and $z'=(z'_1,\ldots,z'_n)$ in the topological closure $\overline{D^{\IV}}$, we define $z_{n+1}$ and $z'_{n+1}$ as in \eqref{eq:def-E-IV}. We then have:
$$F_z(E(z'))=F(z)\wedge E(z') =
\det 
\begin{pmatrix}
 &  I_{n}  & & 0 & z'_1 \\
 &    & & \vdots & \vdots\\
     \overline{z_1} & \ldots & \overline{z_n} & -\overline  z_{n+1}& z'_{n+1}   \\
     0			  & \ldots  &	0		&  1 & 1
 \end{pmatrix}, 
$$
$$
=-\det 
\begin{pmatrix}
 &  I_{n}  & & 0 & z'_1 \\
  &    & & \vdots & \vdots\\
  	0			  & \ldots  &			&  1 & 1\\
      \overline{z_1} & \ldots & \overline{z_n} & -\overline  z_{n+1}& z'_{n+1}
\end{pmatrix}, 
$$
which, using the formula of the determinant of a block matrix gives
\begin{eqnarray}\label{eq:duality-F-IV}
F_z(E(z'))=-({\overline{z}_{n+1}}+z'_{n+1}) +\sum_{k=1}^{n} \bar z_k z'_k.
\end{eqnarray}

\section{Singular value decomposition and the geometry of classical 
Hermitian symmetric spaces}\label{sec:SVD}

In the next section we will describe the Hilbert metric on classical 
bounded symmetric domains. To give an explicit expression of the 
distance between two points, we will need a normalisation of pairs of 
points. It is given by the  singular value decomposition of complex 
matrices. The {\it singular values} of a complex matrix 
$Z \in M_{p,q}(\C)$, where $p\leqslant q$ are the square roots of the 
eigenvalues of  $ZZ^*$. Denote these by $\sigma_1(Z)\geqslant \cdots\geqslant \sigma_p(Z)\geqslant 0$.

For any tuple $\underline s=(s_1,\cdots,s_r)$ of real numbers (where $r\leqslant p$,), we denote by $\Sigma_{p,q}(\underline s)$ the matrix
\[ \Sigma_{p,q}(s_1,\ldots, s_r) := \begin{pmatrix} s_1 & & &0&\ldots&0\\
	&\ddots & &\vdots& &\vdots\\
	&&s_r&0&\ldots &0\\
	0&\ldots&0&0&\ldots & 0\\
	\vdots& &\vdots&\vdots&&\vdots\\
	0&\ldots&0&0&\ldots&0\end{pmatrix}\in M_{p,q}(\C).
\]
The following theorem is well-known:
\begin{theorem*}[Singular values decomposition]
Let $Z\in M_{p,q}(\C)$ be a matrix, where $p\leqslant q$. Denote by $\Sigma_Z$ the matrix $\Sigma_{p,q}(\sigma_1(Z),\cdots, \sigma_p(Z))$. Then, there exist a pair of matrices 
$(U,V)\in U(p)\times U(q)$  such that
\begin{equation}\label{eq:svd}
 Z = U\Sigma_Z V^*.
\end{equation}
\end{theorem*}

The action of the group $SU(p,q)$ on the bounded symmetric domain 
$D^\I_{p,q}$ is given by
\[
    \begin{pmatrix} A & B \\ C & D\end{pmatrix}
    \cdot Z = (AZ +B)(CZ+D)^{-1}.
\]
Note that the condition that $Z\in D^\I_{p,q}$ implies that the matrix 
$(CZ+D)$ is invertible (see Section (2.2) of Chapter 4 in 
\cite{mok-MetricRigidity} for details). An important point is that this action 
is transitive. 

Indeed, given $Z\in D^\I_{p,q}$, we can write its singular 
value decomposition as $Z= U\Sigma_ZV^*$. By a direct computation, we see 
that $I_q-Z^*Z$ is conjugate by $V^*$ to  $I_q-\Sigma_Z^*\Sigma_Z = I_q-\Sigma_Z^2$. 
Therefore, the condition that $Z\in D^I_{p,q}$ implies that the singular 
values of $Z$ satisfy $0\leqslant \sigma_i(Z)<1$. 
Denoting $\tau_i=\sqrt{1-\sigma_i^2}$, we can thus consider the two 
square matrices of respective sizes $p$ and $q$ given by
\[
	T(Z) = U \begin{psmallmatrix} 1/\tau_1 &&\\&\ddots &\\&&1/\tau_p\end{psmallmatrix}U^*\quad \textrm{ and } T'(Z) = V \begin{psmallmatrix} 1/\tau_1 &&&&&\\&\ddots &&&&\\&&1/\tau_p&&&\\ &&&1&&\\ &&&&\ddots&\\ &&&&&1\end{psmallmatrix}V^*.
\] 
Then, the block-matrix
\[M = \left(\begin{array}{@{} c|c @{}}
	T(Z) & -T(Z)Z \\
	\hline
	-T'(Z)Z^* & T'(Z) 
\end{array}\right)\] 
belongs to $\SU(p,q)$ and satisfies $M\cdot Z =0$; see e.g. 
\cite[Chap. 4, Section 2.2]{mok-MetricRigidity}.

We now want to describe the action on pair of points, analogously to 
the case of the bidisc, see Section \ref{ssec:bidisc3}. The action is 
not anymore transitive, as it preserves any invariant metric. However, 
one can send any pair to a pair $(0,S)$ where $S$ is a matrix of the form 
$\Sigma_{p,q}(s_1,\ldots s_p)$ with $s_1\geq \cdots \geq s_p$, as stated 
in the following basic lemma:
\begin{lemma}\label{lem:pairs_in_D-I}
Let $Z$ and $Z'$ be two points in $D^\I_{p,q}$. There exists a 
unique ordered $p$-uple $1>\sigma_1\geq \cdots \geq \sigma_p\geq 0$ 
and an element $M\in\SU(p,q)$ such that: 
\[
    M\cdot Z = 0 \quad \textrm{ and } \quad M\cdot Z' = 
    \Sigma_{p,q}(\sigma_1,\ldots,\sigma_p).
\]
\end{lemma}
\begin{proof}
We have seen above that there exist $M_0$ sending $Z$ to $0$. 
Now, the stabilizer of $0$ is the product 
$\mathrm{S}(\mathrm{U}(p)\times \mathrm{U}(q))$ and a pair $U,V$ 
in this stabilizer acts on $W'$ by $(U,V)\cdot W' = UW'V^*$. 
The lemma then follows directly from the singular value decomposition.
\end{proof}

The procedure carries on to other types, giving the descriptions in 
the next three lemmas. We begin with Type $\II$, for which the description
depends on the parity of $n$:
\begin{lemma}\label{lem:pairs_in_D-II}
	Let $W$ and $W'$ be two points in $D^\II_{n}$. Then there exists 
    a unique ordered $p$-uple  $1>\sigma_1\geq \cdots \geq \sigma_p\geq 0$ 
    (where $p=\lfloor n/2 \rfloor$),  and an element $M\in\SO^*(n)$ 
    such that: 
	\[
        M\cdot W = 0 \quad \textrm{ and } \quad M\cdot W' =
            \left\{ \begin{array}{l} 
                \begin{bmatrix} w(\sigma_1) & & \\
                    &\ddots & \\
                     & & w(\sigma_p)
                \end{bmatrix}
                \mbox{ if $n$ is even }\\
                \\ 
                \begin{bmatrix} 
                    w(\sigma_1) & & & \\
                    & \ddots &  &\\ 
                    & & w(\sigma_p) & \\ 
                    & & & 0
                \end{bmatrix}
                \mbox{ if $n$ is odd }
            \end{array}\right.,\]
	where 
        \[
            w(\sigma) = \begin{bmatrix}
                            0 & \sigma \\
                            -\sigma & 0
                        \end{bmatrix}.\] 
	The singular values of $W'$ are 
    $1>\sigma_1 = \sigma_1\geq \sigma_2 = 
        \sigma_2 \geq \cdots \geq \sigma_p = \sigma_p$, 
        with an additional $0$ when $n$ is odd.	
\end{lemma}
The case of Type $\III$ is easier to describe:
\begin{lemma}\label{lem:pairs_in_D-III}
	Let $W$ and $W'$ be two points in $D^\III_{n}$. Then there exists 
    a unique ordered $n$-uple  $1>\sigma_1\geq \cdots \geq \sigma_n\geq 0$ 
    and an element $M\in\Sp(n)$ such that: 
	\[
        M\cdot W = 0 \quad \textrm{ and } \quad M\cdot W' = 
        \Sigma_{n,n}(\sigma_1,\ldots,\sigma_n).
    \]
\end{lemma}
Eventually, we describe Type $\IV$:
\begin{lemma}\label{lem:pairs_in_D-IV}
	Let $W$ and $W'$ be two points in $D^\IV_{n}$. Then there 
    exists a unique couple  $1>\sigma_1\geq \sigma_2\geq 0$ 
    and an element $M\in\SO^0(2,n)$ such that: 
	\[
        M\cdot W = 0 \quad \textrm{ and } 
        \quad M\cdot W' = (\sigma_1,i\sigma_2,0,\ldots,0) \in M_{1,n}(\C).
    \]
\end{lemma}

We now have reviewed the needed material on the bounded symmetric domains and how the 
singular values decomposition helps understand the geometry. We can proceed 
with the description of the Hilbert metric.

\section{The Hilbert metric}\label{section:Hilbert}
We recall here the general definition of Hilbert metric given in 
\cite{FalbelGuillouxWill}. For $V$ a vector space,
we denote by $\P$ the projective space $\P(V)$ and by $\P'$ its dual
$\P(V^\vee)$. Given a point $\omega\in \P$ and a form $\phi$ in $\P'$ 
denote by $\bomega$ and $\bphi$ any lift to $V$ and $V^\vee$.
We consider two non-empty subsets $\Omega \subset \P$ and
$\Lambda \subset \P'$ such that the following non-vanishing condition holds 
(compare with Section \ref{ssec:bidisc3}).
\begin{equation}\label{eq:condi-subset}\forall \omega \in \Omega,\, \forall \phi \in \Lambda \ \ \ \ \bphi(\bomega)\neq 0.\end{equation}
Geometrically, each point in $\P'$ represents a hyperplane in $\P$, and condition
\eqref{eq:condi-subset} means that $\Omega$ is disjoint from 
all hyperplanes defined by points in $\Lambda$. 
We call such a pair $(\Omega,\Lambda)$ {\it admissible}.
We then define a cross-ratio between two forms and two points:
\begin{definition} \label{def:crossratio}
Let $(\phi,\phi',\omega,\omega')\in\Lambda^2\times\Omega^2$. The cross-ratio
$[\phi,\phi',\omega,\omega']$ is defined as
\begin{equation}
\label{eq:comput-crossratio}
    [\phi,\phi',\omega,\omega']=
    \dfrac{\bphi(\bomega)\bphi'(\bomega')}{\bphi(\bomega')\bphi'(\bomega)}
\end{equation}
\end{definition}
We assume now that the set $\Lambda$ is compact. The generalized Hilbert metric
is defined by the following:
\begin{definition}\label{def:def-dist}
The Hilbert semi-metric $d_\Lambda$ is the function 
defined on $\Omega \times \Omega$ by
\begin{equation}\label{eq:def-dist}
d_\Lambda (\omega,\omega') = \ln\left(\max \left\{
            \left|\left[\phi,\phi',\omega,\omega'\right]\right| \textrm{ for }
            \phi,\phi'\textrm{ in } \Lambda\right\}\right).
\end{equation}
\end{definition}
It is always a semi-metric \cite{FalbelGuillouxWill}: it is symmetric, 
respects the triangular inequality, but may not separate points.

In the previous section, we have defined for each type of classical bounded 
symmetric domain two embeddings $E:D\to \P$ and $F:\partial_SD \to \P'$ 
(for the appropriate dimension of projective spaces). Moreover, for all $x\in D$ and $z\in\partial_SD$, the form $F(z)$ 
does not vanish at $E(x)$ by Proposition \ref{prop:proj-embed-all}. In other terms:
\begin{corollary}
	For any classical Hermitian symmetric space 
    $D=D^\I_{p,q}$, $D^\II_n$, $D^\III_n$, $D^\IV_n$, 
    the pair $(E(D),F(\partial_S D))\subset \P\times \P'$ is admissible.
\end{corollary}
We can then define the Hilbert metric:
\begin{definition}
	For any classical bounded symmetric domain 
    $D=D^\I_{p,q}$, $D^\II_n$, $D^\III_n$, $D^\IV_n$, 
    we define the semi-metric $d_D(x,x'):= d_{F(\partial_S D)}(E(x),E(x'))$.
\end{definition}
The goal of this section is to prove the following by actually computing 
the Hilbert metric and its associated Finsler infinitesimal metric in each case:
\begin{theorem}\label{thm:main}
	For any classical bounded symmetric domain 
    $D=D^\I_{p,q}$, $D^\II_n$, $D^\III_n$, $D^\IV_n$, 
    the semi-metric $d_D$ is an actual metric, invariant by 
    $\mathrm{Aut}(D)$ and comes from a Finsler infinitesimal metric.

	This Finsler metric is neither the Carathéodory nor the Bergman infinitesimal metric if the symmetric space is not of rank one.
\end{theorem}
We prove this theorem through a case by case analysis. Note that the invariance is already obtained by \cite{FalbelGuillouxWill} and we shall use it for computations.

\subsection{Proof of Theorem \ref{thm:main} for type $\I$}\label{sec:formula-type-I}

Here, the space $V$ is $\Lambda^q(\C^{p+q})$. Denote by $D = D^\I_{p,q}$ and
let $\Lambda\subset  \P'$ be the image under $F^\I$ of the Shilov boundary $\partial_SD$. Pick $x$ and $x'$ in $D$ and denote by $\omega:=E^\I(x)$, $\omega'=E^\I(x')$.
Following Definition  \ref{def:def-dist}, the Hilbert  metric is defined by
$$
d_D (x,x') = \ln\left(\max \left\{
            \left|\left[\phi,\phi',\omega,\omega'\right]\right| \textrm{ for }
            \phi,\phi'\textrm{ in } \Lambda\right\}\right).
$$
Note first that $d_D$ is clearly invariant under projective 
transformations preserving $\Lambda$.
In order to compute it and its associated Finsler metric, we use the 
projective invariance and Lemma \ref{lem:pairs_in_D-I}: we can assume 
that $x$ is the origin $0\in M_{p,q}(\C)$ and 
$x' = \Sigma_{p,q}(\sigma_1,\ldots,\sigma_p)$ is a diagonal matrix, 
where $1> \sigma_1 \geq \cdots \geq \sigma_p\geq 0$. 
Denoting by
\[  
    O:=E^\I(0)= \begin{bmatrix} 0 \\
							I_q
			\end{bmatrix} \quad 
    \textrm{ and } 
    \omega':=E^\I(x')=\begin{bmatrix} x'\\
							I_q
			\end{bmatrix},
\] 
and using \eqref{eq:identification-dual}, \eqref{eq:duality} 
and  Proposition \ref{prop:duality-explicit}, we get that,  
for $Z\in \partial_SD$ and $x\in D$
$$F^\I_Z(E^\I(x))=F^\I(Z)\wedge E^\I(x) = \det(I_p-xZ^*).$$

The cross-ratio $[\varphi,\varphi',\omega,\omega']$  becomes 
thus after normalization of the pair of points:
\begin{eqnarray}
[F^\I_Z,F^\I_{Z'},O,\omega'] = &\dfrac{\det({I_p-{x'Z'^*}})}{\det(I_p-{x'Z^*})}\nonumber\\
= & \dfrac{\det(I_p-\Sigma_{p,q}(\sigma_1,\ldots, \sigma_p){Z'^*})}{\det(I_p-\Sigma_{p,q}(\sigma_1,\ldots, \sigma_p){Z^*})}\label{eq:cross-ratio-HC}.
\end{eqnarray}
We can now use \eqref{eq:cross-ratio-HC} to obtain an explicit value 
for the Hilbert metric. 
\begin{proposition}\label{prop:dist-DI-normal}
Assume $D=D^\I$. The Hilbert distance to the origin of 
$x= \Sigma_{p,q}(\sigma_1,\cdots,\sigma_p)$,
where $1>\sigma_1\geqslant\cdots \geqslant \sigma_p\geqslant 0$, satisfies
\begin{equation}
d_D(0,x) = \sum_{i=1}^p\ln\left(\dfrac{\sigma_i+1}{1-\sigma_i}\right)\label{eq:dist-DI-normal}
\end{equation}
\end{proposition}
Before proving Lemma \ref{prop:dist-DI-normal}, we note that 
\eqref{eq:dist-DI-normal} implies in particular that, if $x\neq 0$ or, equivalently, if $\sigma_1>0$, then $d_D(0,x)\neq 0$. In particular, $d_D$ separates 
points and it is an actual metric (not just a semi-metric). 

\begin{proof}[Proof of Proposition \ref{prop:dist-DI-normal}]
Let us denote $\Sigma_{p,q}(\sigma_1,\ldots,\sigma_p)=\Sigma$ for short.
In order to compute $d_D(0,x)$, we need to understand the range of 
$|\det(xZ^*-I_p)|$ when $Z$ varies in the Shilov boundary, that is $Z$ satisfies 
$ZZ^*=I_p$. Note that the singular values of such a matrix $Z$ are all equal to 
$1$. Now, for a given $Z$ in the Shilov boundary, let $A=xZ^*$. First we 
note that in this situation, we have $AA^*= \Sigma\Sigma^*$ (this follows 
from the fact that all singular values of $Z$ are 1 by a straightforward computation).
This fact implies first that $xZ^*$ and $x$ have the same 
singular values, and secondly that $A$ is a normal matrix. 
As a consequence, the singular values of $A$ are exactly the
absolute value of its eigenvalues. It follows from this discussion that 
if the eigenvalues of $A$ are $\lambda_1,\cdots \lambda_p$ (numbered  
so that $\sigma_i(A)=|\lambda_i|$), then  we have
\[|\det(xZ^*-I_p)|=\prod_{i=1}^p|\lambda_i-1|\]
In particular, we observe that
\begin{equation}\prod_{i=1}^p (1-\sigma_i)\leqslant|\det(xZ^*-I_p)|\leqslant 
    \prod_{i=1}^p(1+\sigma_i).\label{eq:double}
\end{equation}
But the right and left-hand side of the double inequality \eqref{eq:double} 
are respectively attained by $|\det(xZ^*-I_p)|$ precisely when 
$Z=Z_+=\Sigma_{p,q}(1,\ldots,1)$ and $Z_-=-Z_+$.
This proves that the max in \eqref{eq:cross-ratio-HC} is obtained precisely when 
$(Z,Z')=(Z_-,Z_+)$, which yields \eqref{eq:dist-DI-normal}:
\begin{eqnarray*}
	d(0,x) & = & \ln\left(\max_{Z,Z'\in\partial_S D^\I_{p,q}} 
        \dfrac{|\det(xZ^*-I_p)|}{|\det(xZ'^*-I_p)|}\right) \\
	& = & \ln\left( \dfrac{|\det(xZ_-^*-I_p)|}{|\det(xZ_+^*-I_p)|} \right) \\
	& = & \sum_{i=1}^p \ln\left(\frac{\sigma_i+1}{1-\sigma_i}\right).
\end{eqnarray*}
This concludes the proof.
\end{proof}

As a straightforward consequence, we obtain the following by differentiating in \eqref{eq:dist-DI-normal}.
\begin{corollary}\label{coro:infinitesimal-metric-I}
Let $\xi\in T_0\Omega\simeq M_{p,q}$ be a tangent vector at the origin $0\in D^\I_{p,q}$, with singular values $\sigma_1\geq \cdots \geq \sigma_p \geq 0$. The Finsler norm of $\xi$  is
$$
    \vert\vert \xi\vert\vert_O = 2(\sigma_1+\cdots + \sigma_p).
$$
\end{corollary}
Corollary \eqref{coro:infinitesimal-metric-I} implies directly that the Hilbert metric is neither the Bergman nor the Carathéodory metric, since these two metrics respectively correspond to the $L^2$ and $L^\infty$ infinitesimal metrics: 
\begin{itemize}
 \item the Carathéodory infinitesimal metric gives ${\vert\vert \xi\vert\vert_C}_O = \max\{\sigma_1,\cdots, \sigma_p\}$ (see \cite{Kobayashi,Suzuki}).
 \item the Bergman infinitesimal metric at the origin is given by a multiple of ${\vert\vert \xi\vert\vert_B}_O = \sqrt{\sigma^2_1 +\cdots + \sigma^2_p}$ depending on normalizations (see \cite{Morita}).
\end{itemize}
This finishes the proof of Theorem \ref{thm:main} in the case of the symmetric spaces $D_{p,q}^\I$.
 
\subsection{Proof of Theorem \ref{thm:main} for type $\II$ and $\III$}\label{sec:formula-type-II-III}

The situation for types $\II$ and $\III$ is analogous. Lemmas \ref{lem:pairs_in_D-II} and \ref{lem:pairs_in_D-III}  provide a normalized form for pairs of points which allows an explicit computation in terms of singular values. The following proposition sums-up the results in that case.
\begin{proposition}\label{prop:Finsler-type-II-III}
Assume $D=D^\II_n$ or $D^\III_n$.  If the singular values of $x\in D$ are 
$1> \sigma_1\geq \cdots \geq \sigma_n \geq 0$, then:
$$d_D(0,x) = \sum_{i=1}^n \ln\left(\frac{\sigma_i+1}{1-\sigma_i}\right).$$
For $\xi\in T_0\Omega\simeq M_n$ a tangent vector at the origin $0\in D$, with singular values $\sigma_1\geq\cdots \geq\sigma_n\geq 0$, the Finsler norm of $\xi$ at the origin is
$$
\vert\vert \xi\vert\vert_O = 2(\sigma_1+\cdots + \sigma_n).
$$
\end{proposition}
Here again, the expression of the Finsler norm shows that the metric 
is different from both the Bergman or Carathéodory ones (see the discussion 
at the end of the previous section).

\begin{proof}
The proof follows along the same lines as for Proposition 
\ref{prop:dist-DI-normal} and Lemma \ref{coro:infinitesimal-metric-I}. 
One verifies that the values of the determinants are as follows.
\begin{itemize}
	\item For type $\III$, not only is $\partial_S D^\III \subset \partial_S D^\I$, 
    but also $\partial_S D^\III$ contains the two points $Z_-$ and $Z_+$ 
    defined in the proof of Proposition \ref{prop:dist-DI-normal}. They are 
    clearly the arguments of the maximum for $D^\III$ as well.
	\item for type $\II$,  $Z_-$ and $Z_+$ do not anymore belong to the 
    Shilov boundary, but one can use instead $Z_\pm = \pm Z_0$, where $Z_0$ 
    is the block diagonal matrix with blocks 
    $\left(\begin{smallmatrix} 0 & -1\\
         1 & 0\end{smallmatrix}\right)$, 
    as in the definition of $D^\II$, see \eqref{eq:type-II}.
\end{itemize}
\end{proof}
Theorem \ref{thm:main} is then proven for Types $\II$ and $\III$.

\subsection{Proof of Theorem \ref{thm:main} for type $\IV$}
\label{sec:formula-type-IV}
We now turn to the last Type $\IV$. We follow the same strategy: we compute 
the distance between any two points by considering the normalized form 
given by Lemma \ref{lem:pairs_in_D-IV}. Recall that $(e_i)_{i=1}^{n+2}$ 
is the canonical basis of $\C^{n+2}$. We are going to use the homogeneous 
coordinates coming from the basis 
$$
    \left(e_1,\ldots,e_n,\frac{e_{n+1}-ie_{n+2}}{\sqrt{2}},
        \frac{e_{n+1}-ie_{n+2}}{\sqrt{2}}
    \right)
$$ 
in which the form $q$ is given by the matrix
$$ J=\begin{bmatrix}I_n& &  \\& & -1\\ & -1 &  \end{bmatrix}.$$
Recall that if $z\in\partial_S(D^{\IV})$ is a point in the Shilov boundary 
and $w\in D^{\IV}$, then  equation \eqref{eq:duality-F-IV} gives a formula for
$F_z(E(w))$.

The following Lemma sums-up the necessary quantities to compute the distance 
between two (normalized) points. The proof is straightforward using the 
descriptions of the Harish-Chandra embeddings at the end of Section 
\ref{section:typeIV}, together with Equation 
\eqref{eq:duality-F-IV}.
\begin{lemma}\label{lem:values}
Let $E$ and $F$ be the maps described by \eqref{eq:def-E-IV} and \eqref{eq:def-F-IV}.  
\begin{enumerate}
 \item The image of the origin $0\in\C^n$ is  
 $$E(0)=[0:\ldots:0:0:1]^T.$$
 \item If $x=(x_1,ix_2,0,\ldots,0)$ with $x_i\in\R$, then 
 $$E(x) = \Bigl[x_1:ix_2:0:\ldots:0:\frac12(x_1^2-x_2^2):1\Bigr]^T.$$
 \item Let $v=e^{i\theta}(v_1,\ldots,v_n)$, where $v_i\in\R$, 
 be a point in the Shilov boundary of $D^{\IV}$. Then
 $$E(v)=\Bigl[e^{i\theta}v_1:\ldots:e^{i\theta}v_n:e^{2i\theta}:1\Bigr]^T.$$
 \item Moreover, the values of linear forms associated to $v$ at the normalized 
 pair of points are given by:
 \begin{eqnarray*}
  F_v(E(0)) & =& -e^{-2i\theta}\\
    F_v(E(x)) & =& 
        e^{-i\theta}\left(v_1x_1+iv_2x_2-\dfrac{x_1^2-x_2^2}{2}e^{i\theta}-e^{-i\theta}\right)
 \end{eqnarray*}
\end{enumerate}
\end{lemma}
In order to compute the value of the Hilbert distance we need to estimate $|F_z(E(x))|$.
\begin{lemma}\label{lem:estim-Fz(Ex)}
Let $z_-$ and $z_+$ be the two points  in the Shilov boundary given by $z_\pm= (\pm\sqrt{2},0,\ldots,0)$. Then, for any $z\in \partial_S D^\IV$ and any $x = (x_1,ix_2,0,\ldots)\in D^{\IV}$ with $x_1\geq x_2$, we have
\[ 1+\frac{x_1^2-x_2^2}{2} -\sqrt{2}x_1 = |F_{z_-} E(x)| \leq |F_{z} E(x)| \leq |F_{z_+} E(x)| = 1+\frac{x_1^2-x_2^2}{2} +\sqrt{2}x_1.\]
\end{lemma}
\begin{proof}
Let us denote by $V_1$ and $V_2$ the quantities 
\begin{eqnarray*}
V_1&=&v_1x_1+iv_2x_2\\
V_2 &=& e^{i\theta}+e^{-i\theta}\dfrac{x_1^2-x_2^2}{2}\\
&=&\cos\theta\left(1+\dfrac{x_1^2-x_2^2}{2}\right)+
    i\sin(\theta)\left(1-\dfrac{x_1^2-x_2^2}{2}\right) 
\end{eqnarray*}
We have seen in Lemma \ref{lem:values} that $F_z(E(x))=e^{-i\theta}\left(V_1-V_2\right)$,
so that we need to estimate $|V_1-V_2|$. Remark that
\begin{itemize}
    \item On the one hand, since $v$ belongs to the Shilov boundary, so that $(v_1,v_2)$ 
    verify $v_1^2+v_2^2\leq 2$. Hence, the point $V_1$ ranges over of the ellipse 
    $\mathcal{E}_1$ which is centered at $0$, and whose axes are horizontal (with half-length $x_1\sqrt{2}$) and vertical (with half-length $x_2\sqrt{2}$).
    \item On the other hand, as $\theta$ varies, $V_2$ ranges over the ellipse 
    $\mathcal{E}_2$ with horizontal axis having half-length 
    $1+(x_1^2-x_2^2)/2$ and vertical axis having half-length  
    $1-(x_1^2-x_2^2)/2$ (see Figure \ref{fig:minmax}). 
    \item The fact that $x\in \mathcal{O}$ gives 
    $$x_1^2+x_2^2<1+\left(\dfrac{x_1^2-x_2^2}{2}\right)^2,$$
    which implies in turn that
    \[
    \begin{aligned}2x_1^2 & = (x_1^2+x_2^2) +(x_1^2-x_2^2)\\ & < 1+\left(\dfrac{x_1^2-x_2^2}{2}\right)^2+(x_1^2-x_2^2)\\ & = \left(1+\dfrac{x_1^2-x_2^2}{2}\right)^2.
    \end{aligned}
    \]
\end{itemize}
As a consequence we have 
$$\sqrt{2}x_1<1+\dfrac{x_1^2-x_2^2}{2},
$$
and analogously we obtain
$$
\sqrt{2}x_2<1-\dfrac{x_1^2-x_2^2}{2}.
$$
This means that $\mathcal{E}_1$ is contained in the interior of 
$\mathcal{E}_2$. The maximum of $|V_1-V_2|$ is thus attained when 
$v_1=\sqrt{2}, v_2=0$ and $\theta=\pi$ (which correspond to two marked 
black points on Figure \ref{fig:minmax}).

The minimum value of $|V_1-V_2|$ is attained either along the horizontal 
axis (for $v_1 = \sqrt{2}$, $v_2 = 0$, $\theta = 0$), or along the vertical 
axis (for $v_1 = 0$, $v_2 = \sqrt{2}$, $\theta = \pi/2$). 
\begin{itemize}
 \item In the first case, the minimum equals $A=1+\dfrac{x_1^2-x_2^2}{2}-x_1\sqrt{2}$, 
 which  after a little rewriting is seen to be equal to 
 $$\left(1-\dfrac{x_1+x_2}{\sqrt{2}}\right)\left(1-\dfrac{x_1-x_2}{\sqrt{2}}\right).$$
 \item In the second case the minimum is equal to $B=1-\dfrac{x_1^2-x_2^2}{2}-x_2\sqrt{2}$, 
 which is equal to 
  $$\left(1-\dfrac{x_1+x_2}{\sqrt{2}}\right)\left(1+\dfrac{x_1-x_2}{\sqrt{2}}\right).$$
\end{itemize}
The last two quantities are positive as values of $|V_1-V_2|$. 
Together with $x_1\geqslant x_2$, this implies that 
$1-(x_1+x_2)/\sqrt{2}\geqslant 0$. Now, the difference is
$$A-B =-\sqrt{2} \left(1-\dfrac{x_1+x_2}{\sqrt{2}}\right)(x_1-x_2)<0$$
So the minimum is attained for $v_1 = \sqrt{2}$, $v_2 = 0$, $\theta = 0$ 
and is equal to $A$.
\end{proof}

\begin{figure}
\begin{center}
\scalebox{0.4}{\includegraphics{Ellipses.eps}}
\end{center}
\caption{The two ellipses $\mathcal{E}_1$ and $\mathcal{E}_2$ and the min and max values for $|V_1-V_2|$.\label{fig:minmax}}
\end{figure}

\begin{proposition}\label{prop:dist-DIV}
	Assume $D=D^\IV_n$. Let $x\in D$ be such that 
    $x = (x_1,ix_2,0\ldots,0)$ with $x_1\geq x_2\geq 0$. Then we have:
	\begin{equation}\label{eq-dist-IV}
        d_D(0,x') = \ln\left(\frac{1+\frac{x_1^2-x_2^2}{2} +
            \sqrt{2}x_1}{1+\frac{x_1^2-x_2^2}{2} -\sqrt{2}x_1}\right).
    \end{equation}
		
	For $\xi = (\xi_1,i\xi_2,0,\ldots,0)\in T_0D$ (with $\xi_1, \xi_2 \in \R$) a tangent 
    vector at the origin $0\in D^\IV_n$, the Finsler norm of $\xi$ at 
    the origin is
	$$
	\vert\vert \xi\vert\vert_0 = 2\sqrt 2 \max(\vert \xi_1\vert,\vert\xi_2\vert).
	$$
\end{proposition}

\begin{proof}
The proof of Proposition \ref{prop:dist-DIV} is straightforward 
from Lemma \ref{lem:estim-Fz(Ex)}: if $x$ is as above, then the 
maximum of the cross-ratio is attained for $z_-=(-x_1\sqrt{2},0,\dots,0)$ 
and $z_+=(x_1\sqrt{2},0,\ldots,0)$  as in Lemma \ref{lem:estim-Fz(Ex)}. 
This gives the expression for the distance announced in \eqref{eq-dist-IV}. 
The Finsler norm is obtained  by differentiating the distance function.  It suffices to consider 
    $\xi_1\geq \xi_2\geq 0$, and we obtain in this case $\vert\vert \xi\vert\vert_0= 2\sqrt 2 \xi_1$.  This implies the formula.
\end{proof}

The last point to prove in the main theorem is that the metric is not equal to the Carathéodory metric.  This follows from the following lemma.

\begin{lemma}For $\xi = (\xi_1,i\xi_2,0,\ldots,0)\in T_0D$ (with $\xi_1, \xi_2 \in \R$) a tangent 
    vector at the origin $0\in D^\IV_n$, the Carathéodory  norm of $\xi$ at 
    the origin is
	$$
	\vert\vert \xi\vert\vert^{car}_0 = \frac{1}{\sqrt{2}} (\vert \xi_1\vert+\vert\xi_2\vert).
	$$
\end{lemma}
\begin{proof}

In order to prove the lemma  we  come back to the initial example: 
the bidisc $U^2$, see Section \ref{ssec:bidisc}. In fact, the bidisc $U^2$ 
identifies with the domain $D^{\IV}_2$ through the map $(z,w)\mapsto \frac{1}{\sqrt{2}}(z+w,i(z-w))$. The Hilbert metric for   $D^{\IV}_2$ given above coincides with the one for $U^2$ given in the introduction, up to a factor $2$. Indeed, a tangent vector $(\xi,\eta)$ at $(0,0)\in U^2$ is sent to the tangent vector  $\frac{1}{\sqrt{2}}(\xi+\eta,i(\xi-\eta))\in T_0D^\IV_2$. 

Given a tangent vector $(\xi,\eta)$ (with $\xi,\eta\in \R$) at $(0,0)$ in the bidisc, we obtain from the previous proposition, for its image $(\xi_1,i\xi_2)=\frac{1}{\sqrt{2}}(\xi+\eta,i(\xi-\eta))\in T_0D^\IV_2$ the value of its norm $2\sqrt 2 \max(\vert \xi_1\vert,\vert\xi_2\vert)=2\max(|\xi+\eta|,|\xi-\eta|)$, which equals $2(|\xi|+|\eta|)$.  

The Carathéodory metric on the bidisc
is, on the other hand different, and equal to $\max(|\xi|,|\eta |)$.  Using the argument of \cite{Kobayashi} section IV.2 example 2,  applying the same change of variables as above, the Carathéodory norm of $(\xi_1,i\xi_2,0,\ldots,0)\in T_0D$ is $\frac{1}{\sqrt{2}} (\vert \xi_1\vert+\vert\xi_2\vert)$.

\end{proof}

\bibliographystyle{alpha}
\bibliography{biblio}

\begin{flushleft}
  \textsc{E. Falbel, A. Guilloux\\
  Institut de Math\'ematiques de Jussieu-Paris Rive Gauche \\
CNRS UMR 7586 and INRIA EPI-OURAGAN \\
 Sorbonne Universit\'e, Facult\'e des Sciences \\
4, place Jussieu 75252 Paris Cedex 05, France \\}
 \verb|elisha.falbel@imj-prg.fr, antonin.guilloux@imj-prg.fr|
 \end{flushleft}
\begin{flushleft}
  \textsc{P. Will\\
  Univ. Grenoble Alpes, CNRS, IF, 38000 Grenoble, France}\\
  \verb|pierre.will@univ-grenoble-alpes.fr|
\end{flushleft}

\end{document}